\documentclass[11pt, oneside]{article}   	
\usepackage{amsmath, amsthm, amsfonts, amssymb}
\usepackage{graphicx}
\usepackage[colorlinks,citecolor=blue,urlcolor=blue]{hyperref}
\newcommand\Ex{{\mathbb E}}

\newcommand\Prob{{\mathbb P}}

\newcommand\bfk{{\mathbf k}}

\newcommand\cC{{\mathcal C}}

\newcommand\cF{{\mathcal F}}

\newcommand\cK{{\mathcal K}}

\newcommand\cP{{\mathcal P}}

\newcommand\cX{{\mathcal X}}

\newcommand\R{{\mathbb R}}

\newcommand\K{{\mathbb K}}

\newcommand\vto{\overset{v}{\to }}

\newcommand\norm[1]{\|#1\|}

\newcommand\bra[1]{\langle #1 \rangle}

\DeclareMathOperator{\image}{im}

\DeclareMathOperator{\conv}{conv}

\newtheorem{theorem}{Theorem}[section]

\newtheorem{lemma}[theorem]{Lemma}

\theoremstyle{definition}
\newtheorem{definition}[theorem]{Definition}

\theoremstyle{remark}

\title{On persistent homology of random \v{C}ech complexes\footnote{This work is partially supported by JST CREST Mathematics (15656429) and by JSPS KAKENHI Grant Numbers JP16K17616}}
\author{Khanh Duy Trinh
\footnote{Research Alliance Center for Mathematical Sciences, Tohoku University, Japan. 
\newline
Email: \texttt{trinh.khanh.duy.a3@tohoku.ac.jp}}
}

\begin{document}
\maketitle

\begin{abstract}      %optional
The paper studies the relation between critical simplices and persistence diagrams of the \v{C}ech filtration. We show that adding a critical $k$-simplex into the filtration corresponds either to a point in the $k$th persistence diagram or a point in the $(k-1)$st persistence diagram. Consequently, the number of points in persistence diagrams can be expressed in terms of the number of critical simplices. As an application, we establish some convergence results related to persistence diagrams of the \v{C}ech filtrations built over binomial point processes.

\medskip

		
\medskip
	
\end{abstract}

\section{\v{C}ech complexes: critical simplices and persistent homology}
Let $\cX \subset \R^N$ be a finite set. For a radius parameter $t \ge 0$, the \v{C}ech complex $\cC_t(\cX)$ is defined by
\[
	\cC_t(\cX) = \left\{\emptyset \neq \sigma \subset \cX : \bigcap_{x \in \sigma} \bar B_t(x) \ne \emptyset \right\},
\]
where $\bar B_t(x) = \{y \in \R^N : \norm{y - x} \le t\}$ denotes the closed ball of radius $t$ centered at $x$ with respect to the Euclidean distance $\norm{\cdot}$. This is an abstract simplicial complex homotopy equivalent to the union of balls $\cup_{x \in \cX} \bar B_t(x)$ (by the nerve lemma). A subset $\sigma \subset \cX$ of cardinality $(k+1)$ is called a $k$-simplex, or simply a simplex. A $0/1/2$-simplex is usually referred to as a vertex, an edge or a triangle, respectively.

When the radius $t$ is small enough, the  \v{C}ech complex $\cC_t(\cX)$ consists of only vertices. As $t$ increases, simplices are added more and more until all simplices have already been included. Topology features such as rings and cavities are created and then disappear as the radius parameter runs from zero to infinity. To see how topology features change in the filtration $\{\cC_t(\cX)\}_{t \ge 0}$, we use a mathematical tool called persistent homology which can be visualized by persistence diagrams. Roughly speaking, a point $(b,d)$, called a birth-death pair, in the $q$th persistence diagram corresponds to a $q$-dimensional hole which appears at time $t = b$, persists during $t \in [b, d)$, and disappears at time $t = d$. %Note that the $q$th persistence diagram is determined by the $q$th persistent Betti numbers $\beta_q^{r,s}(\K), (0 \le r \le s < \infty)$, which count the number of points $(b,d)$ such that $b \le r$ and $d > s$. 

%Let us first see how simplcies are added into the filtration $\{\cC_t := \cC_t(\cX)\}_{t \ge 0}$.  
Let 
\begin{equation}\label{f}
	f(\sigma) = \inf \{t \ge 0 : \sigma \in \cC_t(\cX)\}
\end{equation}
be the birth time of a simplex $\sigma$. It may happen that several simplices are added at the same time. However, when the points $\cX$ satisfy some mild conditions, the \v{C}ech filtration has a very nice structure in the sense that either of the following holds
\begin{itemize}
	\item[(0)]	$\cC_t = \cC_t^-$, where $\cC_t^-= \cup_{s < t} \cC_s$;
	\item[(i)]	$\cC_t = \cC_t^- \sqcup \{\alpha\}$;
	\item[(ii)]	$\cC_t = \cC_t^- \sqcup [\tau, \sigma]$, for $\tau \subsetneq \sigma$, where $[\tau, \sigma]= \{\alpha : \tau \subset \alpha \subset \sigma\}$ is called an interval.
\end{itemize}
Here `$\sqcup$' denotes the disjoint union.

In case (i), a $k$-simplex $\alpha$ is called a critical $k$-simplex and its birth time $f(\alpha)$ a critical value. When a critical $k$-simplex is added, either one generator is created in the $k$th homology group or one generator disappears in the $(k-1)$st homology group. In other words, adding a critical $k$-simplex at time $t=c$ corresponds either  to a point in the $k$th persistence diagram with birth time $b = c$, or to a point in the $(k-1)$st persistence diagram with dead time $d = c$. It turns out that homology groups do not change when an interval $[\tau, \sigma]$ with $\tau \subsetneq \sigma$  is added. That is to say, the cases (0) and (ii) do not effect homology groups. From those, the number of points in the $q$th persistence diagram can be calculated from the number of critical $k$-simplicies ($k \le q$). The aim of this paper is to give a rigorous proof of the arguments above.

As an application, we study persistence diagrams of the \v{C}ech filtration built over binomial point processes. Let $\{X_n\}_{n \ge 1}$ be an i.i.d.~(independent identically distributed) sequence of $\R^N$-valued random variables with bounded probability density function $\kappa(x)$ (with respect to the Lebesgue measure on $\R^N$). The union of the first $n$ points, $\cX_n = \{X_1, X_2, \dots, X_n\}$ is called a binomial point process. Random \v{C}ech complexes in this setting have been extensively studied. Refer to \cite{BK} for a survey. Let $PD_q(n^{1/N} \cX_n)$ be the $q$th persistence diagram of the \v{C}ech filtration built over 
\[
n^{1/N} \cX_n = \{n^{1/N}X_1, n^{1/N}X_2, \dots, n^{1/N}X_n\}
\] 
and let $\xi_{q,n}$ be its counting measure
\[
	\xi_{q,n} = \sum_{(b,d) \in PD_q(n^{1/N} \cX_n)} \delta_{(b,d)}.
\]
Here $\delta_{(b,d)}$ is the Dirac measure at $(b,d)$. Then for $1 \le q \le N-1$, as a random measure on 
\[
	\Delta = \{(b,d) : 0 \le b < d < \infty\},
\]
almost surely, $n^{-1}\xi_{q,n}$ converges vaguely as $n \to \infty$ to a nonrandom measure $\nu_{q,\kappa}$ which can be expressed in terms of the limiting measure in the homogeneous Poisson point process setting \cite{Goel-2018}.

The limiting behavior of critical simplices has been studied \cite{Bobrowski-Adler-2014,Bobrowski-Mukherjee-2015}. In order to deal with critical simplices, we need the following assumption.

{\bf Assumption A.}
The probability density function $\kappa(x)$ has convex and compact support $S$, and
\[
	0 < \inf_{x \in S} \kappa(x) \le \sup_{x \in S} \kappa(x) < \infty.
\]
Under Assumption A, the weak law of large numbers and the central limit theorem for the  number of critical $k$-simplices in $\{\cC_t(n^{1/N}\cX_n)\}_{t \ge 0}$ have been established \cite{Bobrowski-Adler-2014}.  We improve in this paper by showing the strong law of large numbers. This result has interesting consequences which we summarize in the following.
\begin{theorem}
Let $1\le q \le N-1$. Then under Assumption A, the following hold.
\begin{itemize}
\item[\rm(i)] Almost surely,
	\[
		\frac{\xi_{q, n} }{n} \to \nu_{q, \kappa} \quad \text{vaguely as $n \to \infty$}.
	\]
%Here $\nu_{k,f}(A) = \int \nu_k(f(x)^{1/d} A) f(x) dx$. 
The total mass $\nu_{q,\kappa}(\Delta)$ does not depend on $\kappa$.

\item[\rm(ii)] Almost surely, 
	\[
		\frac{\xi_{q, n}(\Delta) }{n} \to M_{N,q} \quad \text{as $n \to \infty$},
	\]
where $M_{N,q}$ is a constant depending only on $N$ and $q$.

\item[\rm(iii)] Almost surely, 
\[
	\frac{1}{n}\sum_{(b, d) \in PD_q(n^{1/N} \cX_n)} \delta_{(d- b)} \to (M_{N,q} - \nu_{q, \kappa}(\Delta)) \delta_0 + \nu_{q, \kappa} \circ pr^{-1} \quad\text{weakly as $n \to \infty$}.
\]
Here $pr \colon \Delta \ni (b,d) \mapsto d-b \in (0, \infty)$.
\end{itemize}
\end{theorem}
That the relation $M_{N,q} = \nu_{q,\kappa}(\Delta)$ holds or not is still open.

\section{Persistent homology of generalized discrete Morse functions}

Let $\bfk$ be a field. For a simplicial complex $\cK$, denote by $C_q(\cK), Z_q(\cK)$ and $B_q(\cK)$ the $q$th chain group, the $q$th cycle group and the $q$th boundary group, respectively. Notations are taken from \cite{HST-2018}.

For a right continuous filtration of simplicial complexes $\K = \{\cK_t\}_{t \ge 0}$ whose persistent homology is assumed to be tame, let
\[
	\beta_q^{r, s} = \beta_q^{r, s}(\K) = \dim \frac{Z_q(\cK_r)}{Z_q(\cK_r) \cap B_q(\cK_s)}, \quad (0 \le r \le s < \infty),
\]
be the $q$th persistent Betti numbers. As a function of $(r,s)$, $\beta_q^{r, s}$ becomes the `distribution function' of the $q$th persistence diagram. In fact, let $\xi_q$ be the counting measure on $\bar \Delta := \{(b,d) : 0 \le b < d \le \infty\}$ defined by
\[
	\xi_q = \sum_{(b,d) \in PD_q(\K)} \delta_{(b,d)},
\]
where $PD_q(\K)$ denotes the $q$th persistence diagram of the persistent homology of the filtration $\K$. Then
\[
	\beta_q^{r,s} = \xi_q([0, r] \times (s, \infty]).
\]

Let $\cK$ be a finite simplicial complex. For $\tau \subset \sigma$, the interval $[\tau, \sigma]$ of simplices is defined as 
\[
	[\tau, \sigma] := \{\alpha : \tau \subset \alpha \subset \sigma\}. 
\]
It contains a single simplex in case $\tau = \sigma$.
Let $V$ be a partition of $\cK$ into intervals. Then a function $f \colon \cK \to \R$ is called a generalized discrete Morse function with generalized discrete gradient $V$ if
\begin{itemize}
	\item[(i)] $f (\tau) \le f(\sigma)$, whenever $\tau \subset \sigma$; and
	\item[(ii)] for $\tau \subset \sigma$, the equality $f(\tau) = f(\sigma)$ holds iff $\tau$ and $\sigma$ belong to the same interval in $V$. 
\end{itemize}
Refer to \cite{Bauer-2017} for the usage of terminologies.

Let $f \colon \cK \to [0, \infty)$ be a generalized discrete Morse function with $f(\{v\}) = 0$ for all vertices $v$. For $t \ge 0$, let 
\[
	\cK_t = f^{-1}([0, t]) = \{\sigma \in \cK : f(\sigma) \le t\}.
\]
Assume that the value of $f$ on each interval (except vertices) is different. Then for $t > 0$, either of the following holds
\begin{itemize}
	\item[(0)]	$\cK_t = \cK_t^-$;
	\item[(i)]	$\cK_t = \cK_t^- \sqcup \{\alpha\}$;
	\item[(ii)]	$\cK_t = \cK_t^- \sqcup [\tau, \sigma]$, for $\tau \subsetneq \sigma$.
\end{itemize}
When a single $k$-simplex $\alpha$ is added at time $t = f(\alpha)$, we call $\alpha$ a critical $k$-simplex and its birth time $f(\alpha)$ a critical value.

Now consider the persistent homology of $ \{\cK_t\}_{t \ge 0}$ which is tame because $\cK$ is finite.  Let $N_k, k \ge 1$, be the number of critical $k$-simplices. Let $N_0$ be the number of vertices in $\cK$. Let $M_q$ be the number of finite points (points $(b,d)$ with $d < \infty$) in $PD_q$. At the end of this section, we show that a critical $k$-simplex with critical value $t = c$ corresponds to either a point in $PD_k$ with $b = c$ or a point in $PD_{k - 1}$ with $d = c$. Consequently, the following relations hold
\begin{align*}
	N_0 &= M_0 + \beta_0 (\cK),\\
	N_1 &= M_0 + M_1 + \beta_1(\cK), \dots,\\
	N_q &= M_{q - 1} + M_{q} + \beta_q(\cK).
\end{align*}
Here $\beta_k(\cK)$ is the $k$th Betti number of $\cK$. Consequently, the following holds.
\begin{theorem}\label{thm:PD-critical}
For $q \ge 0$,
\[
	M_q = N_q - N_{q -1} + \cdots + (-1)^q N_0 - (\beta_q(\cK) - \beta_{q-1}(\cK) + \cdots + (-1)^q \beta_0(\cK)).
\]
\end{theorem}

The above relations will be proved through several lemmata. In what follows, the notation $\alpha^{(p)}$ is used to indicate that $\alpha$ is a $p$-simplex. 

\begin{lemma}\label{lem:add-one-simplex}
Assume that $\cK\subset \tilde\cK$ are simplicial complexes with  $\tilde \cK = \cK \sqcup \alpha^{(p)}$. Then  
\[
	Z_{p - 1} (\tilde \cK) = Z_{p - 1}(\cK), \quad B_p(\tilde \cK) = B_p(\cK),
\]
and,
\[
\begin{cases}
	B_{p - 1}(\tilde \cK) = B_{p - 1}(\cK), \quad Z_p(\tilde \cK) \cong Z_p(\cK) \oplus \bfk, &\text{if } \partial_p(\bra{\alpha}) \in B_{p - 1}(\cK),\\
	B_{p - 1}(\tilde \cK) = B_{p - 1}(\cK) \oplus \bfk \partial_p(\bra{\alpha}) , \quad Z_p(\tilde \cK) = Z_p(\cK), &\text{if } \partial_p(\bra{\alpha}) \not\in B_{p - 1}(\cK).
\end{cases}
\]
Here $\partial_p \colon C_p(\cK) \to C_{p-1}(\cK)$ is the $p$th boundary operator and the notation `$\oplus$' denotes the direct sum of vector spaces.
\end{lemma}

\begin{proof}
Since $\tilde \cK = \cK \sqcup \alpha^{(p)}$, it follows that 
\[
	C_k(\tilde \cK) = \begin{cases}
		C_k(\cK) \oplus \bfk \bra{\alpha}, &\text{if } k = p,\\
		C_k(\cK), &\text{if }k \neq p.
	\end{cases}
\]
Note that $B_{p}(\cK) = \image \partial_{p+1}$ and $Z_p(\cK) = \ker \partial_p$.
Then it is clear that 
\[
	Z_{p - 1} (\tilde \cK) = Z_{p - 1}(\cK), \quad B_p(\tilde \cK) = B_p(\cK).
\]

In case $\partial_p(\bra{\alpha}) \in B_{p - 1}(\cK)$, then  $B_{p - 1}(\tilde \cK) = B_{p - 1}(\cK)$. Consequently, by the rank--nullity theorem,
\begin{align*}
	\dim Z_p(\tilde \cK) &= \dim C_p(\tilde \cK) - \dim B_{p-1}(\tilde \cK) \\
	&= 1 + \dim C_p( \cK) - \dim B_{p-1}( \cK) = 1 +  \dim Z_p(\cK),
\end{align*}
which implies
\[
	\quad Z_p(\tilde \cK) \cong Z_p(\cK) \oplus \bfk.
\]

Conversely, when $\partial_p(\bra{\alpha}) \notin B_{p - 1}(\cK)$, then clearly $B_{p - 1}(\tilde \cK) = B_{p - 1}(\cK) \oplus \bfk \partial_p(\bra{\alpha})$. By using the rank--nullity theorem again, it follows that $Z_p(\tilde \cK) = Z_p(\cK)$. The proof is complete.
\end{proof}

\begin{lemma}\label{lem:add-two}
Assume that $\cK\subset \tilde\cK$ are simplicial complexes with $\tilde \cK = \cK \sqcup \{\alpha^{(p)} \subset \beta^{(p+1)}\}$. Then only the $p$th boundary and the $p$th cycle group change in a way that
\[
	B_{p}(\tilde \cK) = B_p(\cK) \oplus \bfk \partial_{p+1}(\bra{\beta}),\quad Z_{p}(\tilde \cK) = Z_p(\cK) \oplus \bfk \partial_{p+1}(\bra{\beta}).
\]
Consequently,
\begin{equation}\label{persistent}
	Z_p(\cK) \cap B_p(\tilde \cK) = B_p(\cK). 
\end{equation}
\end{lemma}

\begin{proof}
Assume that $\tilde \cK = \cK \sqcup \{\alpha^{(p)} \subset \beta^{(p+1)}\}$. Then 
\[
	C_k(\tilde \cK) = \begin{cases}
		C_k(\cK) \oplus \bfk \bra{\alpha}, &\text{if } k = p,\\
		C_{k}(\cK) \oplus \bfk \bra{\beta}, &\text{if } k = p+1,\\
		C_k(\cK), &\text{if }k \notin \{p, p+1\}.
	\end{cases}
\]
Now, by the definition of the boundary operator, we have
\begin{equation}\label{boundary-of-beta}
	\partial_{p+1}(\bra{\beta}) = \pm \bra{\alpha} + \sum_{\sigma_j^{(p)} \subset \beta; \sigma_j^{(p)} \neq \alpha } \pm \bra{\sigma_j}.
\end{equation}
Clearly,  the sum in the above equation belongs to $C_p(\cK)$. By taking the boundary operator again, we deduce that $\partial_p(\bra{\alpha}) \in B_{p-1}(\cK)$ because $\partial_p \circ \partial_{p+1} = 0$. Thus, similarly as in the proof of Lemma~\ref{lem:add-one-simplex}, we obtain that
\[
	B_{p - 1}(\tilde \cK) = B_{p - 1}(\cK), \quad Z_p(\tilde \cK) \cong Z_p(\cK) \oplus \bfk.
\]
In this case, we can write explicitly
\[
	Z_p(\tilde \cK) = Z_p(\cK) \oplus \bfk \partial_{p+1}(\bra{\beta}).
\]
Note that adding $\{\alpha^{(p)} \subset \beta^{(p+1)}\}$ does not effect the $(p-1)$st cycle group, that is, $Z_{p-1}(\tilde \cK) = Z_{p-1}(\cK)$.

Next we consider the boundary operator $\partial_{p+1}$,
\[
	\partial_{p + 1} \colon C_{p+1} (\cK) \to C_p(\cK), \quad \partial_{p + 1} \colon C_{p+1} (\cK) \oplus \bfk \bra{\beta} \to C_p(\cK) \oplus \bfk \bra{\alpha}.
\]
It follows from the expression of $\partial_{p+1}(\bra{\beta})$ in \eqref{boundary-of-beta} that  
\[
	B_{p}(\tilde \cK) = B_p(\cK) \oplus \bfk\partial_{p+1}(\bra{\beta}).
\] 
Then $Z_{p+1}(\tilde \cK) = Z_{p+1}(\cK)$ by comparing their dimensions using the rank--nullity theorem. The second statement is an easy consequence of the first one.
The proof is complete.
\end{proof}

%
%\begin{lemma}
%	Let $B, Z$ and $W$ be vector spaces with $B\subset Z$ and $W\cap Z = \{0\}$. Then 
%	\[
%		Z \cap (B \oplus W) = B.
%	\] 
%\end{lemma}

We have shown that 
\[
		Z_k(\cK) \cap B_k(\tilde \cK) = B_k(\cK), (k \ge 0),
\]
if $\tilde \cK = \cK \sqcup \{\alpha^{(p)} < \beta^{(p+1)}\}$. (It'd better to write $\alpha^{(p)} < \beta^{(p+1)}$ here instead of $\alpha^{(p)} \subset \beta^{(p+1)}$.) By induction, the above relation still holds, if $\cK$ and $\tilde \cK$ is connected by a sequence of simplicial complexes $
	\cK = \cK_0 \subset \cK_1 \subset \cdots \subset \cK_n = \tilde \cK
$
with $\cK_i = \cK_{i - 1} \sqcup \{\alpha_i^{(p_i)} < \beta_i^{(p_i+1)}\}$. In this case, $\tilde \cK$ is said to collapse onto $\cK$. We claim that $\tilde \cK$ collapses onto $\cK$, if $\tilde \cK = \cK \sqcup [\tau, \sigma]$, for $\tau \subsetneq \sigma$. Indeed, choose an arbitrary vertex $x \in \sigma \setminus \tau$, and partition the interval $[\tau, \sigma]$ into pairs $\{\alpha \setminus \{x\}, \alpha \cup \{x\}\} $ with noting that 
\[
	\{\alpha \setminus \{x\}, \alpha \cup \{x\}\} = \begin{cases}
		\{\alpha \setminus \{x\}, \alpha\}, &\text{if $x \in \alpha$},\\
		\{\alpha, \alpha \cup \{x\}\}, &\text{if $x \not\in \alpha$}.
	\end{cases}
\]
Then arranging the pairs in a suitable order yields the desired result.

Theorem~\ref{thm:PD-critical} follows directly from the following result.

\begin{theorem}\label{thm:critical}
Let $\K=\{\cK_t\}_{t\ge0}$ be a right continuous filtration of simplicial complexes whose persistent homology is assumed to be tame.
\begin{itemize}
\item[\rm(i)]
Assume that only one $p$-simplex is added at time $c \in (u, v]$, that is, $\cK_c = \cK_{c}^- \sqcup \alpha^{(p)},$ and $ \cK_t = \cK_t^{-}$ for $t \in (u,v] \setminus \{c\}$. Then either $PD_{p-1}$ has only one point in the region $\{(b,d) : d \in (u, v]\}$ with $b \le u$ and $d = c$, or $PD_p$ has only one point in the region $\{(b,d) : b \in (u, v]\}$ with $b = c$ and $d > v$.

\item[\rm(ii)] Assume that for $u \le r \le s \le v$, 
\[
	Z_p(K_r) \cap B_p(K_s) = B_p(K_r).
\]
Then $PD_p$ has no point in the region 
	\[
		\{(b,d) : b \in (u, v] \text{ or } d \in (u,v]\}.
	\]
\end{itemize}
\end{theorem}

%\begin{figure}[th]
%\includegraphics[width=.45\textwidth]{PD_qm.eps}~~~
%\includegraphics[width=.45\textwidth]{PD_q.eps}
%\caption{Either $PD_{q-1}$ (left) or $PD_{q}$ (right) occurs}
%\end{figure}

\begin{proof}
(i)~Since $\cK_c = \cK_{c}^- \sqcup \alpha^{(p)}$, it follows from Lemma~\ref{lem:add-one-simplex} that there are two cases to deal with.

{CASE 1:} $Z_p(\cK_t) = Z_p(\cK_u), t \in (u,v]$, and $B_{p - 1}(\cK_t) = B_{p-1}(\cK_u), t \in (u, c)$ and $B_{p - 1}(\cK_t) \cong B_{p - 1}(\cK_u) \oplus \bfk, t \in [c,v]$.

(a)~Let us consider $PD_{p-1}$. For fixed $r \in [u, v]$, the function of $s\in [r,v]$
		\begin{align*}
			\beta_{p-1}^{r, s}(\K) &= \dim \frac{Z_{p-1}(\cK_r)}{Z_{p-1}(\cK_r) \cap B_{p-1}(\cK_s)} \\
			&= \dim \frac{Z_{p-1}(\cK_v)}{Z_{p-1}(\cK_v) \cap B_{p-1}(\cK_s)} \\
			&=	\dim Z_{p-1}(\cK_v) - \dim B_{p-1}(\cK_s),  
		\end{align*}
decreases by one at $s = c$, (if $r < c$). Therefore, there is only one point $(b,d)$ with $d \in (u,v]$. Moreover, the point has $b \le u$ and $d = c$.
	
(b)~For $PD_{p}$, see (ii).

 {CASE 2:} $B_{p - 1}(\cK_t) = B_{p - 1}(\cK_u), t \in (u,v]$,  and $Z_p(\cK_t) = Z_p(\cK_u), t \in  (u,c)$ and $Z_p(\cK_t) = Z_p(\cK_u) \oplus \bfk, t \in  [c, v]$.

	(a)~For $PD_{p-1}$, also see (ii).
	
	(b)~Consider $PD_{p}$. Recall that $B_p(\cK_v) = B_p(\cK_u)$. For fixed $s \in [u, v]$, the function of $r \in [u,s]$
		\[
			\beta_p^{r, s}(\K) = \dim \frac{Z_p(\cK_r)}{Z_p(\cK_r) \cap B_p(\cK_s)} = \dim Z_p(\cK_r) - \dim B_p(\cK_u), \quad r \in [u, s]
		\]
increases by one at $r = c$, (if $s\ge c$). Therefore, there is only one point $(b,d)$ with $b \in (u,v]$. Moreover, the point has $b = c$ and $d > v$. The proof of (i) is complete.

\noindent (ii)~For $u \le r \le s \le v$, it follows from the relation $Z_p(K_r) \cap B_p(K_s) = B_p(K_r)$ that
	\[
		\beta_p^{r,s}(\K) = \dim \frac{Z_p(K_r)}{Z_p(K_r) \cap B_p(K_s)}  = \dim \frac{Z_p(K_r)}{B_p(K_r)},
	\]
from which the conclusion follows. Theorem~\ref{thm:critical} is proved.
\end{proof}

\section{\v{C}ech complexes and generalized discrete Morse theory}

\begin{definition}
A finite set $\cX \subset \R^N$ is in \emph{general position} if for every $\cP \subset \cX$ of at most $N + 1$ points,
\begin{itemize}
	\item[(i)]	$\cP$ is affinely independent, and 
	\item[(ii)] no point of $\cX \setminus \cP$ lies on the smallest circumsphere of $\cP$.
\end{itemize}
	 
\end{definition}

Assume that the points $\cX$ are in general position. Then the function $f$ defined in \eqref{f} is a generalized Morse function (associated with some generalized discrete gradient $V$) \cite{Bauer-2017} . Assume further that the value of $f$ on each interval in $V$ is different. Then for each $t> 0$, it holds that
\[
	\cC_t = \cC_t^-, \text{or} \quad  \cC_t = \cC_t^{-} \sqcup \{\alpha\}, \text{or} \quad \cC_t = \cC_t^- \sqcup[\tau, \sigma] \text{ with } \tau \subsetneq \sigma,
\] 
where recall that $\cC_t = \cC_t(\cX)$ is the \v{C}ech complex with parameter $t$.
From which, the following result follows from Theorem~\ref{thm:PD-critical}.

\begin{theorem}\label{thm:number-of-point-in-PD}
Let $\cX \subset \R^N$ be a set of $n$ points  in general position. Assume that the birth time of each interval is different.
	Let $N_k, (k =1, \dots, N)$ be the number of critical $k$-simplices  in $\{\cC_t(\cX)\}_{t \ge 0}$. Then
	\begin{align*}
		&\#PD_0 = n,\\
		&\#PD_1 = N_1 - (n - 1),\\
		&\dots\\
		&\#PD_q = N_q - N_{q - 1} + \cdots + (-1)^{q-1} N_1 + (-1)^q (n-1), \quad 1 \le q \le N-1.
	\end{align*}
Here $\#PD_q$ denotes the number of points in the $q$th persistence diagram of $\{\cC_t(\cX)\}_{t\ge0}$.
\end{theorem}

We conclude this section with some remarks on critical simplices.
A $k$-simplex $\alpha$ is critical, if  $f(\tau) <  f(\alpha) < f(\sigma)$, for any $\tau \subsetneq \alpha \subsetneq \sigma$. A criterion for this recovers the concept of critical simplices in \cite{Bobrowski-Adler-2014}. For a set $\alpha$ of $(k+1)$ points in general position in  $\R^N$, let 
\begin{align*}
	S(\alpha) &= \text{the unique $(k-1)$-dimensional sphere containing $\alpha$},\\
	C(\alpha) &= \text{the center of $S(\alpha)$ in $\R^N$},\\
	R(\alpha) &= \text{the radius of $S(\alpha)$},\\
	B(\alpha) (\bar B (\alpha)) &= \text{the open (closed) ball in $\R^N$ with radius $R(\alpha)$ centered at $C(\alpha)$},\\
	\conv^\circ(\alpha) &= \text{the open $k$-simplex spanned by the points in $\alpha$}.
\end{align*}

\begin{lemma}[\cite{Bobrowski-Adler-2014}]
Assume that a finite set $\cX \subset \R^N$ is in general position. Then a $k$-simplex $\alpha$ is critical in the filtration $\{\cC_t(\cX)\}_{t \ge 0}$, if 
\[
	C(\alpha) \in \conv^\circ(\alpha), \quad 
	\cX \setminus \alpha \notin \bar B(\alpha).
\]
\end{lemma}

\section{Persistence diagrams of random \v{C}ech complexes}
Recall that $\{X_n\}_{n\ge 1}$ is an i.i.d.~sequence of $\R^N$-valued random variables with bounded probability density function $\kappa(x)$. We consider persistence diagrams of the \v{C}ech filtration built over $n^{1/N}\cX_n = \{n^{1/N}X_1,n^{1/N}X_2, \dots, n^{1/N}X_n\}$. Note that the assumption in Theorem~\ref{thm:number-of-point-in-PD} holds almost surely.

\subsection{Strong law of large numbers for the number of critical simplicies}

Let $	h_k(\sigma)$ be the indicator function of $k$-simplices with $
C(\sigma) \in \conv^\circ(\sigma)$. Then $\sigma$ is a critical $k$-simplex in $\{\cC_t(\cX)\}_{t \ge 0}$, if and only if $h_k(\sigma) = 1$ and 
$
	\cX \setminus \sigma \notin \bar B(\sigma).
$
Let $N_k(\cX_n)$ be the number of critical $k$-simplicies in $\{\cC_t(\cX_n)\}_{t\ge0}$. Then under Assumption A, it holds that \cite{Bobrowski-Adler-2014},
\[
	\frac{\Ex[N_k(\cX_n)]}{n} \to  \gamma_{N,k},
\]
where
\[
	 \gamma_{N,k} = \frac{1}{(k+1)!} \int_{(\R^N)^k}  h_{k}(\{0,y_1,\dots, y_k\}) e^{-\omega_N R(\{0,y_1,\dots, y_k\})^N} dy_1 \cdots dy_k,
\]
with $\omega_N$ the volume of the unit ball in $\R^N$. Some exact values of $\gamma_{N,k}$ were calculated in \cite{Bobrowski-Mukherjee-2015}
\begin{align*}
	&\gamma_{2,1} = 2,	\quad \gamma_{2,2} = 1,\\
	&\gamma_{3,1} = 4, \quad \gamma_{3,2}=3\left(1 + \frac{\pi^2}{16}\right), \quad \gamma_{3,3} = \frac{3 \pi^2}{16}.
\end{align*}
The weak law of large numbers and the central limit theorem for $N_k(\cX_n)$ were established \cite{Bobrowski-Adler-2014}. Here we show that the strong law of large numbers holds.

\begin{theorem}\label{thm:SLLN-critical}
	Under Assumption A, for $k = 1,2,\dots, N$,
	\[
		\frac{N_k(\cX_n)}{n} \to \gamma_{N, k} \quad \text{almost surely as $n \to \infty$}.
	\]
\end{theorem}

The almost sure convergence follows from the following general result. A detailed proof is left to the reader.

\begin{theorem}[SLLN]\label{thm:SLLN}
 Let $\{X_n\}_{n \ge 1}$ be an i.i.d.~sequence of $\R^N$-valued random variables. Denote by $\cX_n = \{X_1, X_2, \dots, X_n\}$ the corresponding binomial processes.
Let $H_n$ be a real-valued functional defined on finite subsets of $\R^N$.
	Assume that for some $p > 2$, 
	\[
		\sup_{n} \Ex[|H_n(\cX_n) - H_n(\cX_{n - 1})|^p] < \infty.
	\]
Then almost surely,
	\[
		\frac{H_n(\cX_n) - \Ex[H_n(\cX_n)]}{n} \to 0 \quad \text{as} \quad n \to \infty.
	\]
\end{theorem}
The proof of Theorem~\ref{thm:SLLN} is given in Appendix~A.

\subsection{Convergence of persistence diagrams---revisited}

Let $1 \le q \le N-1$. Let $\cP^\lambda$ be a homogeneous Poisson point process in $\R^N$ with density $\lambda > 0$. Denote by $\cP^\lambda_L$ the restriction of $\cP^\lambda$ on $[-\frac{L^{1/N}}2, \frac{L^{1/N}}2)^N$. Then almost surely, as $L \to \infty$,
\[
	\frac{PD_q(\cP_L^\lambda)}{L} \vto \nu_q^\lambda.
\]
Here `$\vto$' denotes the vague convergence of measures and $PD_q$ has the same meaning with its counting measure. By a scaling property of homogeneous Poisson point processes, we can deduce that $\nu_q^\lambda (A)= \lambda \nu_q^1(\lambda^{1/N}A)$, for bounded measurable set $A \subset \Delta$. Note that $\nu_q^\lambda$ has full support \cite{HST-2018}.

For binomial point processes, under a weaker assumption than the boundedness assumption here, it was shown in \cite{Goel-2018} that almost surely, as $n \to \infty$,  
\[	
	\frac{\xi_{q, n}}{n} =\frac{PD_q(n^{1/N}  \cX_n)}{n}  \vto \nu_{q,\kappa},
\]
where 
\[
	\nu_{q, \kappa} (A) = \int \nu^{\kappa(x)}_q (A) dx  = \int \nu_q^1 (\kappa(x)^{1/N} A) \kappa(x) dx.
\]
In particular, $\nu_{q, \kappa} (\Delta) = \nu_q^1(\Delta)$.

Now, as a direct consequence of Theorem~\ref{thm:number-of-point-in-PD} and Theorem~\ref{thm:critical}, it follows that
\[
	\frac{\xi_{q, n}(\Delta)}{n} = \frac{\# PD_q(\cX_n)}{n} \to M_{N,q}\quad \text{almost surely as } n\to \infty,
\]
where 
\[
	M_{N,q} = \gamma_{N,q} - \gamma_{N, q-1} + \cdots + (-1)^{q-1} \gamma_{N,1} + (-1)^q.
\]

Here is the main result in this random part.
\begin{theorem}\label{thm:intervals}
	Let $\zeta_{q,n}$ be the random measure on $[0, \infty)$ defined by
	\[
		\zeta_{q,n} = \frac{1}{n} \sum_{(b, d) \in PD_q(n^{1/N}\cX_n)} \delta_{(d - b)}.
	\]
Then almost surely, as $n\to \infty$, $\zeta_{q,n}$ converges weakly to a measure $\mu_{q,\kappa}$, meaning that for any bounded continuous function $f \colon [0, \infty) \to \R$,  almost surely as $n \to \infty$
\[
	\int f(x) d\zeta_{q,n}(x) = \frac{1}{n}\sum_{(b, d) \in PD_q(n^{1/N}\cX_n)}  f(d - b) \to \int f(x) d\mu_{q,\kappa}(x).
\]
Here 
\[
	\mu_{q,\kappa} = (M_{N,q} - \nu_{q,\kappa}(\Delta))\delta_0 + \nu_{q,\kappa} \circ pr^{-1}, \quad (pr\colon \Delta \ni (b,d) \mapsto d-b).
\]
\end{theorem}

We remark that for $\alpha > 0$, the persistent sum 
\[
	\frac{1}{n} \sum_{(b, d) \in PD_q(n^{1/N}\cX_n)} (d - b)^\alpha
\]
converges almost surely to a finite limit \cite{Divol-2018}. Therefore, the result in the above theorem holds for any continuous function of polynomial growth.

Theorem~\ref{thm:intervals} follows directly from the following deterministic result. A function $F \colon \Delta \to \R$ vanishes at infinity if for every $\varepsilon > 0$, there is a compact $K\subset \Delta$ such that $|f(x)| < \varepsilon$, for $x \in \Delta \setminus K$.
\begin{lemma}
\begin{itemize}
\item[{\rm(i)}] Assume that the sequence of finite measures $\{\mu_n\}$ converges vaguely to $\mu$ as $n \to \infty$ and that 
	\[
		\sup_n \mu_n(\Delta) < \infty.
	\]
Then for any continuous function $F$ vanishing at infinity, it holds that 
	\[
		\int_\Delta F d\mu_n \to \int_\Delta F d\mu \text{ as } n \to \infty.
	\]

\item[{\rm(ii)}] Assume further that there is a sequence of increasing positive numbers $\{r_k\}$ tending to infinity such that 
	\[
		\lim_{k \to \infty}\limsup_{n \to \infty} \mu_n(\{(b,d) \in \Delta : d \ge r_k \}) = 0.
	\]
	Then for any  bounded continuous function  $f \colon [0,\infty) \to \R$ with $f(0) = 0$,
	\[
		\int_\Delta f(d-b) d\mu_n \to \int_\Delta f(d-b) d\mu \text{ as } n \to \infty.
	\]

\item[{\rm(iii)}] Consequently, for any bounded continuous function $f \colon [0,\infty) \to \R$,
	\[
		 \int_\Delta f(d-b) d\mu_n \to (M - \mu(\Delta)) f(0) + \int_\Delta f(d-b) d\mu  \text{ as } n \to \infty,	
	\]
provided that $\mu_n(\Delta) \to M$ as $n \to \infty$.
\end{itemize}
\end{lemma}

The lemma contains some ideas taken from \cite{Divol-2018}.
It can be proved by a standard approximation method so that the proof is omitted. 
Note that the diagonal $\{b=d\}$ plays a role as infinity in the topology of $\Delta$. When $\mu_n$ converges vaguely to $\mu$ and $\mu_n(\Delta) \to \mu(\Delta)$, then $\mu_n$ converges weakly to $\mu$. However, since we only assume that $\mu_n(\Delta) \to M$, where $M$ may not equal $\mu(\Delta)$, some mass could escape to infinity. 

To prove Theorem~\ref{thm:intervals}, it remains to show the condition (ii) in the above lemma. However, that condition is a consequence of the strong law of large numbers for critical simplices in the thermodynamic regime (cf.~\cite{Bobrowski-Adler-2014}).

\appendix
\section{The strong law of large numbers for functionals on binomial point processes}
\renewcommand{\thesection}{A}

In this section, we prove Theorem~{\rm\ref{thm:SLLN}}.
\begin{proof}[Proof of Theorem~{\rm\ref{thm:SLLN}}]
	For fixed $n$, let us estimate $\Ex[|H_n(\cX_n) - \Ex[H_n(\cX_n)]|^p]$.
 Set $Z = H_n(\cX_n)$. Define a martingale sequence $\{M_i\}_{i = 0}^n$ as 
\[
	M_i = \Ex[Z | \cF_i],	\quad  i = 0, 1, \dots, n.
\]
Here $\cF_0 = \{\emptyset, \Omega\}$ and $\cF_i = \sigma(X_j : j \le i)$ for $i =1,2,\dots, n$. Let $\langle Z \rangle$ denote the quadratic variation
\[
	\langle Z \rangle = \sum_{i = 1}^n (M_i - M_{i - 1})^2. 
\]
Then Burkholder's inequalities (last line in page 518 of \cite{Boucheron-2005}) imply that
\begin{equation}\label{Burkholder}
	\|Z - \Ex[Z]\|_p \le (p-1) \|\sqrt{\bra{Z}}\|_p.
\end{equation}

Let $\{X_i'\}_{i = 1}^n$ be an independent copy of $\{X_i\}_{i = 1}^n$. Let 
\[
	Z_i' = H_n(X_1, \dots,X_{i - 1}, X_i', X_{i+1}, \dots X_n), \quad \check Z_i = H_n(X_1, \dots,X_{i - 1}, X_{i+1}, \dots X_n).
\]
Observe that 
\[
	M_i - M_{i - 1} = \Ex[Z - Z_i' | \cF_i].
\]
Hence,
\begin{align*}
	\Ex[|M_i - M_{i - 1}|^p] \le \Ex[|Z - Z_i'|^p] &\le 2^{p -1 } ( \Ex[|Z - \check Z_i|^p] +  \Ex[|Z_i' - \check Z_i|^p]) \\
	&= 2^p \Ex[|H_n(\cX_n) - H_n(\cX_{n - 1})|^p].
\end{align*}
Here the first inequality follows from Jensen's inequality for conditional expectation.

From the definition of $\bra{Z}$, we have 
\[
	\norm{\bra{Z}}_{p/2} \le \sum_{i = 1}^n \norm{(M_i - M_{i - 1})^2}_{p/2} =  \sum_{i = 1}^n \norm{M_i - M_{i - 1}}_{p}^2. 
\]
The last sum is bounded by $n C_p$, for some constant $C_p$, by the assumption. Together with the inequality~\eqref{Burkholder}, it follows that 
\[
	\|Z - \Ex[Z]\|_p \le (p-1) \|\sqrt{\bra{Z}}\|_p =  (p-1) \norm{\bra{Z}}_{p/2}^{1/2} \le (p-1)C_p^{1/2} n^{1/2}.
\]
Therefore for some constant $D_p > 0$,
\[
	\Ex[|H_n(\cX_n) - \Ex[H_n(\cX_n)]|^p] \le D_p n^{p/2},
\]
from which 
\[
	\Prob\left( \frac{|H_n(\cX_n) - \Ex[H_n(\cX_n)]|}{n} \ge \varepsilon \right)  \le \frac{D_p}{\varepsilon^p } n^{-p/2}. 
\]
Then the desired almost sure convergence is a consequence of the Borel--Cantelli lemma. The proof is complete.
\end{proof}

%
%\begin{footnotesize}
%\bibliographystyle{spmpsci}
%%\bibliographystyle{acmtrans-ims}
%\bibliography{bpp-a}
%\end{footnotesize}

\end{document}